\documentclass[a4paper,12pt]{article}

\usepackage[letterpaper,margin=0.85in]{geometry}

\usepackage{amsmath,color}
\usepackage{amsthm}
\usepackage{amsfonts}
\usepackage{amssymb}

\def\R{\mathbb{R}}

\theoremstyle{plain}
\newtheorem{theorem}{Theorem}[section] 

\newtheorem{lemma}[theorem]{Lemma}

\theoremstyle{definition}

\theoremstyle{remark}

\newtheorem{rem}[theorem]{Remark}

\def\cA{\mathcal{A}}

\def\cI{\mathcal{I}}

\def\cO{\mathcal{O}}

\def\cX{\mathcal{X}}

\def\Re{{\textnormal{Re}}}

\def\Id{{\textnormal{Id}}}

\def\txtc{{\textnormal{c}}}
\def\txtd{{\textnormal{d}}}
\def\txte{{\textnormal{e}}}
\def\txti{{\textnormal{i}}}

\def\txtD{{\textnormal{D}}}

\def\ra{\rightarrow}
\def\I{\infty}

\newcommand{\be}{\begin{equation}}
\newcommand{\ee}{\end{equation}}
\newcommand{\benn}{\begin{equation*}}
\newcommand{\eenn}{\end{equation*}}
\newcommand{\bea}{\begin{eqnarray}}
\newcommand{\eea}{\end{eqnarray}}
\newcommand{\beann}{\begin{eqnarray*}}
\newcommand{\eeann}{\end{eqnarray*}}

\title{Scaling Laws and Warning Signs\\ for Bifurcations of SPDEs}
\author{Christian Kuehn\footnote{Faculty of Mathematics, Technical University of 
Munich, Boltzmannstr.~3, 85747 Garching b.~Munich, Germany}~~and~Francesco 
Romano\footnote{Faculty of Mathematics, Technical University of 
Munich, Boltzmannstr.~3, 85747 Garching b.~Munich, Germany \&
Ludwig-Maximilians-Universit\"at, Elite Graduate Course Theoretical and Mathematical Physics, 
Theresienstr. 37, 80333, Munich, Germany.}}

\begin{document}
\maketitle

\begin{abstract}
Critical transitions (or tipping points) are drastic sudden changes observed in many
dynamical systems. Large classes of critical transitions are associated to systems,
which drift slowly towards a bifurcation point. In the context of stochastic
ordinary differential equations (SODEs), there are results on growth of variance
and autocorrelation before a transition, which can be used as possible warning signs in
applications. A similar theory has recently been developed in the simplest setting
also for stochastic partial differential equations (SPDEs) for self-adjoint 
operators in the drift term. This setting leads to real discrete spectrum and 
growth of the covariance operator via a certain scaling law. In this paper, we 
develop this theory substantially further. We cover the cases of complex 
eigenvalues, degenerate eigenvalues as well as continuous spectrum. This provides a 
fairly comprehensive theory for most practical applications of warning signs for 
SPDE bifurcations. 
\end{abstract}

\medskip

\section{Introduction}
\label{sec:intro}

In many areas of science we frequently observe events that appear rather 
abruptly. Some examples are epileptic 
seizures~\cite{McSharrySmithTarassenko,MormannAndzejakElgerLehnertz} 
and asthma attacks~\cite{Venegasetal} in medicine, 
market collapses in economics~\cite{HuangKouPeng,MayLevinSugihara}, 
epidemic outbreaks~\cite{OReganDrake,KuehnZschalerGross}, engineering system 
failures~\cite{CotillaSanchezHinesDanforth}, and 
population/habitat changes in ecology~\cite{Carpenteretal1,CarpenterBrock}. 
Although these \emph{critical transitions} seem 
- a priori - unrelated, there are many unifying features. The events happen rather 
\emph{fast} after a long period of \emph{slow} change, there are special 
\emph{thresholds} or \emph{tipping points} to be crossed, and \emph{stochastic 
fluctuations} are always present. Using stochastic fluctuations to estimate 
presence of, and the distance to, a tipping point has been a successful strategy
already proposed by Wiesenfeld in 1985~\cite{Wiesenfeld1} and tested in the
context of chemical experiments~\cite{KramerRoss}. One exploits
that the main deterministic driving forces near bifurcation are weakened (also 
known as \emph{critical slowing down} or \emph{intermittency}~\cite{GH}) 
and measures the thereby relatively amplified noisy fluctuations. This strategy 
has been (re-)discovered also in many application areas recently, mainly 
in ecology~\cite{CarpenterBrock,Schefferetal} and climate 
science~\cite{Lentonetal,Schefferetal}. Yet, to actually obtain predictive
power of warning signs is often highly non-trivial from a 
practical~\cite{DitlevsenJohnsen} as well as 
statistical~\cite{ZhangHallerbergKuehn,BoettingerHastings} viewpoint.\medskip 

Therefore,
a detailed mathematical theory must be developed to understand the assumptions,
limitations, and opportunities of warning signs for critical transitions better.
For systems modelled by stochastic ordinary differential equations (SODEs), 
a detailed theory can be found in~\cite{KuehnCT2}; see also~\cite{BerglundGentz}
for relevant background.
However, if we discard all spatial components, we may miss important aspects of
the theory, which could also be very important in practical 
applications~\cite{Dakosetal1,DonangeloFortDakosSchefferNes}. This leads one 
to consider stochastic partial differential equations (SPDEs), where warning
signs have only been investigated so far for propagation failure of travelling
waves numerically~\cite{KuehnFKPP} and with a combination of analytical/numerical
methods for stationary patterns in~\cite{GowdaKuehn}. The work~\cite{GowdaKuehn} 
is our main starting point. It focuses on system of the form
\be
\label{eq:SPDEintro}
\begin{array}{lcl}
\txtd u &=& Lu+f(u,p)~\txtd t+\sigma B ~\txtd W,\\
\txtd p &=& \varepsilon g(u,p)~\txtd t,
\end{array}
\ee
where $(x,t)\in\cI\times [0,\I)$, $\cI$ is an interval, $L$ is a 
spatial differential operator, $u=u(x,t)$, 
$p=p(x,t)$, the nonlinearities $f$ and $g$ are sufficiently smooth maps,
$W=W(x,t)$ is a space-time generalized Wiener process, $B$ is a given linear 
operator, $0<\sigma\ll1$ controls the
noise level and $0<\varepsilon\ll1$ is the time-scale separation between
the fast $u$-variable and the slow $v$-variable; see also Section~\ref{sec:back}
for the technical setting. Suppose $f(0,p)=0$ so that
$u_*\equiv 0$ is a homogeneous steady state for any $p$ 
for the deterministic ($\sigma=0$) partial differential equation (PDE). 
The local stability of $u_*\equiv 0$ is determined by studying the operator 
\benn
A=A(p):=L+\txtD_u f(0,p),
\eenn
where $\txtD_u$ is the Fr\'echet derivative and we have to pick a function
space to obtain a well-defined spectral problem. The
basic idea to induce a critical transition in the fast-slow 
SPDE~\eqref{eq:SPDEintro} is that the slow dynamics
\benn
\partial_t p = \varepsilon g(0,p)
\eenn
changes so that for some $p$, say $p<0$, we obtain $\textnormal{spec}(A(p))$ is 
contained in $\{z:\textnormal{Re}(z)<0\}$ while for some other $v$, say $p>0$,
the spectrum contains parts in $\{z:\textnormal{Re}(z)>0\}$. In particular,
this means the fast PDE dynamics
\benn
\partial_t u= Lu+f(u,p)
\eenn
undergoes a bifurcation at $p=0$ as $p$ is varying~\cite{Kielhoefer}. Since 
it is very difficult to control the interplay between $\sigma$, $\varepsilon$ 
and the location of $\textnormal{spec}(A(p))$~\cite{KuehnCurse}, the first natural 
approximation is to consider the fast subsystem singular limit $\varepsilon=0$ 
and just view $p$ as a parameter~\cite{KuehnBook}. In~\cite{GowdaKuehn} this
situation is considered for the linearized problem   
\be
\label{eq:SPDEintro1}
\txtd U = A(p)U~\txtd t+\sigma B \txtd W,\qquad U=U(x,t),
\ee
see also Section~\ref{sec:back}. Several further key assumptions are made
in~\cite{GowdaKuehn}:
\begin{itemize}
 \item[(GK1)] $\textnormal{spec}(A(p))$ contains eigenvalues with multiplicity one; 
 \item[(GK2)] $A(p)$ is self-adjoint;
 \item[(GK3)] the noise term is independent of $p$;
 \item[(GK4)] $\cI$ is a compact interval.
\end{itemize}  
Under these assumptions one can show~\cite{GowdaKuehn} that the covariance operator 
$\textnormal{Cov}(u)$ diverges, when projected on certain Fourier modes
as $p\ra 0^-$. One can also determine and explicit asymptotic scaling power law 
in $p$. These results are a natural generalization to SPDEs for the well-known 
fast-slow SODE setting~\cite{BerglundGentz6,KuehnCT2}.\medskip 

In this paper, we manage to drop and/or generalize all the assumptions (GK1)-(GK4).
We are going to allow for degenerate Jordan blocks lifting (GK1). We also 
consider complex eigenvalues and parameter-dependent noise thereby removing 
(GK2)-(GK3). Furthermore, we are going to consider essential spectrum frequently 
arising for differential operators on unbounded domains. In this context we consider
rather general classes of linear operators $A$. These results are a major 
generalization in contrast to classical differential operators on bounded
domains as in (GK4). The last generalization may
look slightly un-natural at first sight but it is crucial as modulation/amplitude
equations~\cite{KirrmannSchneiderMielke} for SPDEs~\cite{Bloemker} are posed 
on unbounded domains. Modulation equations can be viewed as normal forms for
local pattern formation~\cite{Hoyle}. 

Our results in this paper show that we can essentially
always expect diverging covariance for \emph{generic} noise terms. Either in the form
\be
\label{eq:res1}
\langle \textnormal{Cov}(u_{k^*}),u_{k^*}\rangle =\cO(h(p))~\text{as $p\ra 0^-$,}
\qquad \lim_{p\ra 0^-}h(p)=+\I
\ee
in the Hilbert space $H$ with inner product $\langle\cdot,\cdot\rangle$ for 
some function $u_{k^*}$ and explicitly computable $h(p)$, or more generally for
essential spectrum in the form
\be
\label{eq:res2}
\lim_{p\ra 0^-}\left\| \textnormal{Cov}(u)\right\|=+\I,
\ee
where $\|\cdot\|$ is a norm on linear operators. If the parametric dependence
is chosen so that the noise degenerates, we show that other behaviours are 
possible. For precise technical statements
we refer to Sections~\ref{sec:bounded}-\ref{sec:unbounded}. In summary, this 
completes the theory of warning signs for SPDEs bifurcating from a homogeneous 
steady state in the vast majority of cases of practical relevance.\medskip 

The paper is structured as follows: In Section~\ref{sec:back} we briefly 
present the mathematical background required for our study. In 
Section~\ref{sec:bounded}, we consider the case of discrete spectrum 
for $A=A(p)$. Here we manage to lift the 
assumptions (GK1)-(GK3) and prove a result of the form~\eqref{eq:res1}.
Then we obtain a result of the form~\eqref{eq:res2} for essential spectrum 
in Section~\ref{sec:unbounded}. The proof shows when
we can characterize the precise scaling laws also for essential 
spectrum as stated in~\eqref{eq:res1}.

\section{Background and Framework}
\label{sec:back} 

Consider an evolution equation on a Hilbert space $H$ of the form
\begin{equation}
\label{eqn:linsysb}
\partial_t U=A(p)U,\qquad U=U(t),~p\in\R,
\end{equation}
for a linear operator $A=A(p): \mathcal{D}(A) \subset H \rightarrow H$. Assume
that $A$ is the infinitesimal generator of a strongly continuous semigroup 
$\txte^{tA}$~\cite{Pazy}. Suppose the steady state $U_*=0$ is stable 
for~\eqref{eqn:linsysb} for $p<0$, i.e., the spectrum of $A(p)$ is contained 
in the left half of the complex plane for $p<0$. Suppose at $p=0$ the spectrum 
crosses the imaginary axis $\txti \R$, so that we have an instability that we
interpret as the linearized problem for the drift part of~\eqref{eq:SPDEintro1}. 

Next, we briefly introduce the framework for SPDEs we need 
from~\cite{DaPratoZabczyk}. Consider a filtered probability 
space $(\Omega, \mathcal{F}, \mathcal{F}_t, \mathbb{P})$ and a non-negative 
self-adjoint trace class operator $Q$ on a Hilbert space $H$. By the spectral 
theorem, $Q$ has a countable orthonormal basis $\{q_k\}$ of eigenfunctions, 
with corresponding eigenvalues $\rho_k \geq 0$ such that $Q q_k = \rho_k q_k$. A 
stochastic process $W$ is a $Q$-Wiener process on $H$ if 
\benn
W_t=\sum_{j=1}^\infty \sqrt{\rho_j} \, q_j \, \beta^j_t,\qquad \textnormal{a.s.},
\eenn
where $\beta^j=\beta^j(t)$ are independent and identically distributed 
$\mathcal{F}_t$-adapted Brownian motions and the series converges in 
$L^2(\Omega, H)$. The identity matrix $I$ is not a trace class operator. 
Nevertheless, one can (uniquely) construct a Wiener process with 
covariance matrix that is not trace class by showing that the series
\begin{equation}
\label{eqn:genwiener}
W_t=\sum_{j=1}^\infty Q^{1/2} \, q_j \, \beta^j_t
\end{equation}
converges in a larger Hilbert space $H_1$ (in particular, $H_1$ has to be 
such that the embedding $J: Q^{1/2} H \rightarrow H_1$ is a Hilbert-Schmidt 
operator, see \cite[Prop.~4.7]{DaPratoZabczyk}). The processes 
defined by convergence of the series~\eqref{eqn:genwiener} are 
called \textit{generalized Wiener processes}. A \textit{cylindrical 
Wiener process} (or space-time white noise) is the generalized 
Wiener process with covariance matrix $I$. One can then define integration 
with respect to $Q$-Wiener processes and generalized Wiener processes. 
A general linear additive-noise SPDE can be written in the following form
\begin{equation}
\label{eqn:linSPDEb}
\txtd U=AU~ \txtd t+ \sigma B ~\txtd W_t \qquad U(0)=U_0,
\end{equation} 
where we assume $B \in L_0^2$ with $L^2_0$ denoting the space of Hilbert-Schmidt
operators~\cite{DaPratoZabczyk} and that $U_0$ is a $\mathcal{F}_0$-measurable 
random variable. An $H$-valued predictable process $\{U(t)\}_{t \in [0,T]}$ is 
called a \emph{mild solution} of~\ref{eqn:linSPDEb} if a.s.
\begin{equation}
\label{eqn:mild}
U(t)=\txte^{tA} U_0 + \sigma \int_0^t \txte^{(t-s)A} B~ \txtd W_s.
\end{equation}
Under the assumptions above, mild solutions are guaranteed to exist 
uniquely \cite[Thm.~5.4]{DaPratoZabczyk}. Since we assume that 
$\txte^{tA}U_0$ decay exponentially for $p<0$ and we always take the limit
$p\ra 0^-$, we directly start on the deterministic steady state from now
on and assume 
\benn
U(0)=U_0\equiv 0. 
\eenn
We have the following expression 
for the covariance operator of the second term in~\eqref{eqn:mild} 
as given in~\cite[Thm.~5.2]{DaPratoZabczyk}
\begin{equation}
\label{eqn:covarianceadd}
V(t):=\textnormal{Cov} \Big ( \sigma \int_0^t \txte^{(t-s)A} B ~\txtd W_s \Big)
= \sigma ^2 \int_0^t \txte^{\tau A} BQB^* \txte^{\tau A^*} ~\txtd \tau,
\end{equation}
where $B^*$ denotes the adjoint of $B$. The asymptotic limit 
$V_\infty:=\lim_{t \rightarrow \infty} V(t)=\lim_{t \rightarrow \infty} 
\textnormal{Cov}(U(t))$ satisfies the Lyapunov equation
\begin{equation}
\langle AV_\infty g,h \rangle + \langle V_\infty A^* g, h \rangle 
= - \sigma ^2 \langle BQB^* g, h \rangle
\end{equation} 
for all $h, g$ such that the expression is well defined; 
see~\cite[Lem.~2.45]{DaPrato1}. Hence, we must study the different
behaviours of $V(t)$, respectively $V_\infty$, as $p\ra 0^-$ to understand 
the scaling of the covariance to leading-order as we approach the transition at 
$p=0$~\cite{GowdaKuehn}.

\section{Discrete Spectrum}
\label{sec:bounded}

We start by considering the problem discrete spectrum, motivated by many 
classical differential operators $A$ on bounded domains. 
Our goal is to generalize the following result already obtained 
in~\cite{GowdaKuehn}:
\begin{theorem}
\label{thm:kuehngowda}
Consider~\eqref{eqn:linSPDEb} with
\benn
A=p\;\Id+\cA
\eenn
where $\cA$ has a discrete real spectrum with eigenvalues $\lambda_k \leq 0$, 
eigenfunctions $u_k$ and that there exists a unique $k^*$ such that 
$\lambda_{k^*} = 0$. Also assume the genericity condition 
$\langle BQB^* u_{k^*}, u_{k^*} \rangle \neq 0$ to be satisfied. 
Then, the covariance operator $V(t)$ satisfies 
\begin{equation}
\Big \langle \lim_{t \rightarrow \infty} V(t) u_k, u_j \Big \rangle 
= -\sigma^2 \frac{\langle BQB^* u_k, u_j \rangle}{2p+ \lambda_k + \lambda_j} 
\quad \forall j,k \in \mathbb{N}
\end{equation}
and in particular
\begin{equation}
\Big \langle \lim_{t \rightarrow \infty} V(t) u_{k^*}, u_{k^*} \Big \rangle 
= \mathcal{O} \left(\frac{1}{p} \right) \; \quad \text{as} \; p \rightarrow 0^-.
\end{equation}
\end{theorem}

\begin{proof}
See \cite{GowdaKuehn}, Proposition 3.1.
\end{proof}

The assumptions on the operator $\cA$ guarantee that the spectrum of 
$A=(p\;\Id + \cA)$ is strictly contained in $(-\I,0)$ for $p < 0$. For $p = 0$, 
the spectrum $\textnormal{spec}(A)$ contains the point $0$, which corresponds to the 
eigenfunction $u_{k^*}$. In the language of dynamical systems, in this 
case the steady state $u_*\equiv 0$ is non-hyperbolic and a center manifold 
$W^\txtc_{\textnormal{loc}}(0)$ appears. Being linear, the center manifold is 
explicitly given by the linear subspace $W^\txtc_{\textnormal{loc}}(0)=
\textnormal{span}\{u_{k^*}\}$. The asymptotic result in 
Theorem~\ref{thm:kuehngowda} can then be restated as saying that the component 
of the covariance operator along the center manifold diverges as the critical 
transition is approached. Hence, this is a very natural first analog to the 
results for SODEs in~\cite{KuehnCT2}.

\subsection{Imaginary eigenvalues}

As a first step, we relax the real discrete spectrum assumption on the operator 
$A$ to obtain:
\begin{theorem}
Consider the SPDE~\eqref{eqn:linSPDEb}, i.e., 
\benn
d U=AU~ \txtd t + \sigma B~ \txtd W. 
\eenn
Suppose $A=A(p)$ has a discrete spectrum with eigenvalues $\lambda_k(p)$ 
with $\Re(\lambda_k(p)) < 0$ for all $k$ and $p<0$, and eigenfunctions 
$u_k$. If $k^*$ is such that $\lambda_{k^*}$ is a purely imaginary eigenvalue 
for $p^*=0$ and the genericity condition $\langle BQB^* u_{k^*}, u_{k^*} 
\rangle \neq 0$ is satisfied, the covariance operator $V(t)$ satisfies 
\begin{equation}
\label{eq:complex1}
\left\langle \lim_{t \rightarrow \infty} V(t) u_k, u_j \right\rangle 
= -\sigma^2 \frac{\langle BQB^* u_k, u_j \rangle}{\lambda_k+ \bar{\lambda}_j} 
\quad \forall j,k \in \mathbb{N}
\end{equation}
where $\bar{\lambda}_j$ is the complex conjugate of $\lambda_j$. In particular,
we find
\begin{equation}
\label{eq:complex2}
\left\langle \lim_{t \rightarrow \infty} V(t) u_{k^*}, u_{k^*} \right\rangle 
= \mathcal{O} \left(\frac{1}{\Re(\lambda_{k^*})} \right) \;\quad \text{ as}
 \; p \rightarrow 0^-.
\end{equation}
\end{theorem}

\begin{proof}
The proof is a calculation using the Lyapunov equation 
\begin{equation*}
\langle AV_\infty g,h \rangle + \langle V_\infty A^* g, h \rangle 
= - \sigma ^2 \langle BQB^* g, h \rangle
\end{equation*}
which holds, in particular, for the eigenfunctions $u_k$. Therefore, 
we obtain 
\begin{align*}
\langle AV_\infty u_k,u_j \rangle + \langle V_\infty A^* u_k, u_j \rangle & 
= - \sigma ^2 \langle BQB^* u_k, u_j \rangle, \\
\Rightarrow\quad \lambda_j \langle V_\infty u_k,u_j \rangle 
+ \bar{\lambda}_k \langle V_\infty 
u_k, u_j \rangle & = - \sigma ^2 \langle BQB^* u_k, u_j \rangle, \\
\Rightarrow\quad (\lambda_j + \bar{\lambda}_k) \langle V_\infty u_k, u_j 
\rangle & = - \sigma ^2 \langle BQB^* u_k, u_j \rangle.
\end{align*}
This proves the first claim~\eqref{eq:complex1}. Setting $k=j$ one has
\begin{align*}
\langle V_\infty u_j, u_j \rangle & = - \sigma ^2 
\frac{\langle BQB^* u_j, u_j \rangle}{2\Re(\lambda_j) }
\end{align*}
so for $j=k^*$ the second claim~\eqref{eq:complex2} also follows.
\end{proof}

\subsection{Jordan blocks}

In the previous sections, we have shown divergence in the variance 
along the component corresponding to eigenfunctions of the operator $A$ 
corresponding to the eigenvalue crossing the imaginary axis. In the 
case $A$ has Jordan blocks, one might ask whether such behavior is 
also observed, when projecting along the generalized eigenfunctions 
$\{u_{k^*}^l \}_{l=1,...,m_{k^*}}$. Here, $m_{k^*}$ denotes the dimension 
of the Jordan block corresponding to $\lambda_{k^*}$. For arbitrary $k$, 
setting $u_k^0:=0$ we have the formula $$Au_k^l=u_k^{l-1} 
+ \lambda_k u_k^l.$$ We find that the variance diverges also along 
generalized eigenfunctions, with the rate of divergence depending on 
the order $l$ of the correspondent generalized eigenfunction.

\begin{theorem}
Consider~\eqref{eqn:linSPDEb} and suppose $A=A(p)$ has a discrete 
spectrum with eigenvalues $\lambda_k$ with $\Re(\lambda_k) < 0$ for 
all $k$ and $p<0$. Further assume that $k^*$ is such that 
$\lambda_{k^*}$ is a purely imaginary eigenvalue for $p^*=0$ 
with generalized eigenfunctions $$\{u_{k^*}^l \}_{l=1,...,m_{k^*}}.$$
If the genericity conditions $\langle BQB^* u_{k^*}^1, u_{k^*}^1 
\rangle \neq 0$ is satisfied, the covariance operator $V(t)$ satisfies 
\begin{equation}
\left\langle \lim_{t \rightarrow \infty} V(t) u_{k^*}^l, u_{k^*}^m 
\right\rangle = \mathcal{O} \left(\frac{1}{\Re(\lambda_{k^*})^{l+m-1}} 
\right) \quad \textnormal{ as } p \rightarrow 0^-
\end{equation}
for each $l,m \geq 1$.
\end{theorem}

\begin{proof}
We aim to prove it by induction on $l+m$. First of all suppose $l+m=2$. 
Then the only non-trivial case is $l=m=1$ and the claim has already been 
proven. Therefore we have the first step for induction. We then assume 
the claim holds for all $l,m$ s.t.~$l+m \leq n$ and we want to prove it 
for all $l,m$ s.t.~$l+m = n+1$. Fix such $l$ and $m$. The Lyapunov equation 
implies
\benn
2 \Re(\lambda_k) \langle V_\infty u_k^l, u_k^m \rangle 
+ \langle V_\infty u_k^{l-1}, u_k^m \rangle 
+ \langle V_\infty u_k^l, u_k^{m-1} \rangle 
= - \sigma ^2 \langle BQB^* u_k^l, u_k^m \rangle. 
\eenn
Using the induction assumption $l+m \leq n$ for on the last two terms on 
the right-hand side, we get
\benn
2 \Re(\lambda_k) \langle V_\infty u_k^l, u_k^m \rangle 
= \cO \left( \frac{1}{Re(\lambda_k)^{l+m-2} } \right) 
- \sigma ^2 \langle BQB^* u_k^l, u_k^m \rangle.
\eenn
Therefore, we may conclude that
\begin{equation*}
\langle V_\infty u_k^l, u_k^m \rangle = \cO 
\left( \frac{1}{\Re(\lambda_k)^{l+m-1} } \right) 
- \sigma ^2 \frac{\langle BQB^* u_k^l, u_k^m \rangle}{2 \Re(\lambda_k)} 
= \cO \left( \frac{1}{\Re(\lambda_k)^{l+m-1} } \right),
\end{equation*}
which proves the claim.
\end{proof}

\subsection{Noise and operator dependent on a parameter}

Another interesting case to study is when both $A$ and $\sigma$ depend 
on the parameter $p$. We expect that, if the noise near the bifurcation 
is too small, the variance does not diverge anymore. Indeed, for constant 
noise we observe that the system exhibits slow recovery when the critical 
transition is approached. If the noise decreases too fast, this could 
balance the critical slowing down and prevent the divergence of the variance. 
We show in the next result that it is enough to guarantee 
$\sigma^2 \gg \lambda_{k^*}$ to avoid such a problem.

\begin{theorem}
Consider the SPDE
\begin{equation}
\txtd U=AU~ \txtd t+\sigma(p) B~ \txtd W.
\end{equation}
Assume that $A=A(p)$ has a discrete spectrum with eigenvalues 
$\lambda_k$ with $\Re(\lambda_k) < 0$ for all $k$ and $p<0$, 
and eigenfunctions $u_k$. Assume $k^*$ is such that $\lambda_{k^*}$ 
is a purely imaginary eigenvalue for $p^*=0$ and set $$\Xi:
=\lim_{p \rightarrow 0^-} \frac{\sigma^2(p)}{2 \lambda_{k^*}(p)}.$$ 
The following holds:
\begin{equation}
\lim_{p \rightarrow 0^-} \langle V_\infty u_{k^*}, u_{k^*} \rangle 
= \langle BQB^* u_{k^*}, u_{k^*} \rangle \; \Xi.
\end{equation}
In particular, if we also assume the genericity condition 
$\langle BQB^* u_{k^*}, u_{k^*} \rangle \neq 0$ we have
\begin{equation}
\langle V_\infty u_{k^*}, u_{k^*} \rangle = \cO \left(\frac{\sigma^2}{\lambda_{k^*}} 
\right) \quad \textnormal{ as } p \rightarrow 0^-.
\end{equation}
\end{theorem}

\begin{proof}
We can replicate the same computations as in the preceding sections to obtain
\begin{align*}
\langle AV_\infty u_k,u_j \rangle + \langle V_\infty A^* u_k, u_j \rangle 
&= - \sigma^2(p) \langle BQB^* u_k, u_j \rangle \\
\Rightarrow \quad 
\langle V_\infty u_k,u_k \rangle &= - \sigma^2(p) 
\frac{\langle BQB^* u_k, u_k \rangle}{2\Re(\lambda_k)}
\end{align*}
Again as before, taking the limit for $p \rightarrow 0^-$ gives the required result.
\end{proof}

The last result significantly generalizes an SODE result for a particular
model equation obtained in~\cite[Sec.~7.5]{KuehnCT2}. It shows that one 
must ensure that the noise source does not interact and/or depend in a 
degenerate way on the distance to the critical transition to be able to
obtain a warning sign.

\section{Continuous Spectrum}
\label{sec:unbounded}

We have shown that extended versions of Theorem~\ref{thm:kuehngowda} still 
hold for general discrete spectra, including both complex eigenvalues and 
Jordan blocks. We also found an asymptotic lower bound on the noise $\sigma$, 
which guarantees the result to hold in presence of non-constant external 
perturbations. To further generalize the results in~\cite{GowdaKuehn}, we 
want to also consider differential operators on unbounded domains. This 
naturally leads into the possibility that $A$ has a continuous spectrum. 
For example, this is the case for the one-dimensional Laplacian, which is 
a fundamental operator in modeling diffusion and will be the starting point 
of our discussion. \\

We remark here the main difference with the previous case: in the discrete 
setting we identified the eigenfunction corresponding to the eigenvalue 
crossing the imaginary axis and showed that the component of the covariance 
along that direction tends to infinity as the critical transition is 
approached. However, if essential spectrum crosses the imaginary axis, there 
exists no eigenfunction. For this reason, we start considering the norm of 
the variance and we show that it diverges to infinity (which is of course a 
weaker result). Later, we prove a stronger result using Weyl's theorem on 
approximating eigenfunctions. 

\subsection{The one-dimensional Laplace operator}

Consider the operator $\partial_{xx}$ on the Sobolev space 
$H^2(\mathbb{R})$. We want to study as a simple starting point 
the following modified stochastic heat equation
\begin{equation}
\label{eqn:laplacian}
dU= (p\;\Id+ \partial_{xx}) U ~\txtd t + \txtd W,\qquad U(0)=U_0,
\end{equation}
where we set $\sigma=1$, $B=\Id$ for simplicity of the exposition.
The Laplacian is a self-adjoint operator on $H^2(\mathbb{R})$. We 
recall that, by the Spectral Theorem~\cite{ReedSimon1}, self-adjoint 
operators are unitarily equivalent to multiplication operators. In 
particular, the Fourier transform 
\benn
\mathcal{F}(h)(k):=\hat{h}(k) := \frac{1}{\sqrt{2 \pi}} \int_\mathbb{R} 
\txte^{-\txti kx} h(x)~ \txtd x
\eenn
unitarily maps $\partial_{xx}$ into the multiplication operator 
of multiplication by $|k|^2$. This can be used to prove the divergence 
of the norm of $V_\infty$ as follows:

\begin{theorem}
\label{thm:laplacian}
Consider the SPDE~\eqref{eqn:laplacian}. Then 
\begin{equation}
\lim_{p \rightarrow 0^-} \Vert V_\infty \Vert_{B(L^2(\mathbb{R}))} = +\infty,
\end{equation}
where $\|\cdot\|_{B(L^2(\mathbb{R}))}$ is the norm on linear operators
induced by the $L^2(\R)$-norm.
\end{theorem}

\begin{proof}
First of all, we compute the covariance using~\eqref{eqn:covarianceadd}
\begin{equation*}
V(t)=\int_0^t S(r)BQB^*S^*(r)~\txtd r = \int_0^t S(r)S^*(r)~\txtd r
= \int_0^t \txte^{2r (p+\partial_{xx})} ~\txtd r. 
\end{equation*}
This implies that in the limit $t \rightarrow \infty$ it holds
$$V_\infty =\int_0^\infty \txte^{2r (p+\partial_{xx})} ~\txtd r.$$
Then we take the norm and obtain
\begin{align*}
\Vert V_\infty \Vert_{B(L^2(\mathbb{R}))} &= 
\Big \Vert \int_0^\infty \txte^{2r (p+\partial_{xx})} ~\txtd r 
\Big \Vert_{B(L^2(\mathbb{R}))} = \sup_{h \in L^2(\mathbb{R}), 
\Vert h \Vert =1} \Big \Vert \int_0^\infty 
\txte^{2r (p+\partial_{xx})} h \; ~\txtd r \Big \Vert_{L^2(\mathbb{R})} \\
&= \sup_{h \in H^2(\mathbb{R}), \Vert h \Vert =1} 
\Big \Vert  \int_0^\infty \mathcal{F}^{-1} \txte^{2r (p- |\cdot|^2)} 
\mathcal{F}(h) ~\txtd r \Big \Vert_{L^2(\mathbb{R})}  \\  
&= \sup_{h \in H^2(\mathbb{R}), \Vert h \Vert =1} 
\Big \Vert \frac{1}{\sqrt{2 \pi} } \int_0^\infty \int_\mathbb{R} 
\txte^{\txti kx} \txte^{2r (p- |k|^2)} \hat{h}(k) ~\txtd k ~\txtd r 
\Big \Vert_{L_x^2(\mathbb{R})}  \\
& \overset{\textnormal{Fubini}}{=} \sup_{h \in H^2(\mathbb{R}), \Vert h \Vert =1} 
\Big \Vert \frac{1}{\sqrt{2 \pi} } \int_\mathbb{R} \txte^{\txti kx} 
\hat{h}(k) \int_0^\infty \txte^{2r (p- |k|^2)}  ~\txtd r  ~\txtd k 
\Big \Vert_{L_x^2(\mathbb{R})}  \\
&\overset{p<0}{=} \sup_{h \in H^2(\mathbb{R}), \Vert h \Vert =1} 
\Big \Vert \frac{1}{2\sqrt{2 \pi}} \int_\mathbb{R} \txte^{\txti kx } 
\frac{\widehat{h}(k)}{p - |k|^2} \; ~\txtd k \Big \Vert_{L_x^2(\mathbb{R})} 
= \\ &= \sup_{h \in H^2(\mathbb{R}), \Vert h \Vert =1} 
\Big \Vert \frac{1}{2} \mathcal{F}^{-1} 
\Big [\frac{\widehat{h}(\cdot)}{p - |\cdot|^2} \Big ] 
\Big \Vert_{L^2(\mathbb{R})}  \\
&= \sup_{h \in H^2(\mathbb{R}), \Vert h \Vert =1} 
\Big \Vert \frac{1}{2} \frac{\widehat{h}(\cdot)}{p - |\cdot|^2} 
\Big \Vert_{L^2(\mathbb{R})} \overset{*}{\geq} \Big \Vert 
\frac{1}{2} \frac{\txte^{- |\cdot|^2/2}}{p - |\cdot|^2} \;  
\Big \Vert_{L^2(\mathbb{R})}.
\end{align*}
In $^*$ we set $h= \txte^{- x^2 /2}$, while Fubini's Theorem can be applied since
\begin{equation*}
|\txte^{\txti kx} \txte^{2r (p- |k|^2)} \hat{h}(k)| \leq \txte^{2rp}
 (1+ |k|^2) \hat{h}(k) \in L^2(\mathbb{R}_r^+ \times \mathbb{R}_k)
\end{equation*}
Applying the limes inferior on both sides of the inequality and 
using Fatou's Lemma yields
\begin{align*}
\liminf_{p \rightarrow 0^-} \Vert V_\infty \Vert_{B(L^2(\mathbb{R}))}^2 
& \geq \liminf_{p \rightarrow 0^-} \Big \Vert \frac{1}{2} \frac{\txte^{-
 |\cdot|^2/2}}{p - |\cdot|^2} \;  \Big \Vert_{L^2(\mathbb{R})}^2  = 
\liminf_{p \rightarrow 0^-} \int_\mathbb{R} \frac{1}{4} \frac{\txte^{- 
x^2}}{(p - x^2)^2} ~\txtd x \geq  \\
& \geq  \int_\mathbb{R} \liminf_{p \rightarrow 0^-} \frac{1}{4}
\frac{\txte^{- x^2}}{(p - x^2)^2} ~\txtd x  = \int_\mathbb{R}   
\frac{\txte^{- x^2}}{ x^4} ~\txtd x = + \infty,
\end{align*}
which concludes the proof.
\end{proof}

We stress that Theorem~\ref{thm:laplacian} uses only some particular 
properties of the Laplacian. It relies on three main ingredients: the 
existence of the diagonalizing map $\mathcal{F}$, its interchangeability 
with the integration via Fubini's Theorem and the divergence of the last 
integral. For all self-adjoint operators, the existence of a diagonalizing 
map is guaranteed by the spectral theorem. Once such map is known it might 
be relatively easy to check that also the other requirements for the proof 
are satisfied. Nevertheless, since there exists no explicit formula for 
the diagonalizing map, whether this holds or not has to be studied case 
by case.

\subsection{Multiplication operators}

As we have remarked, a possible generalization of the preceding result 
can be obtain considering general self-adjoint operators. This involves 
applying the spectral theorem and diagonalizing the operator. Being the 
diagonalizing map not known, it is hard to give a formal statement that 
includes all the self-adjoint operators. Instead, we will assume the 
operator to be already diagonalized: namely, we consider multiplication 
operators. Such operators can be characterized as follows:

\begin{theorem}[Structure of multiplication operators~\cite{EngelNagel,Pazy}]
\label{thm:sathm}
Let $\cX$ be a metric space and $\mu$ a positive measure on the Borel 
sigma-algebra of $\cX$ such that $\mu (\Lambda) < \infty $ for any 
bounded Borel set $\Lambda \subset \cX$. For a (possibly unbounded) 
measurable function $f : \cX \rightarrow \mathbb{R}$, the linear 
operator $T_f$ in $L^2 (\cX, \mu)$ defined by 
\benn
(T_f u)(x) := f(x)u(x), \; \; \; D(T_f) = \{ u \in L^2(\cX, \mu) | 
\; fu \in L^2(\cX,\mu) \} 
\eenn
is self-adjoint. Its spectrum coincides with the essential range of 
$f$ and its point/discrete spectrum is given by
\benn
\textnormal{spec}_p(T_f) = \{ \mu(f^{-1}({\lambda})) > 0 \}.
\eenn 
\end{theorem}

Consider now the following stochastic evolution equation on $H=L^2(\cX, \mu)$
\begin{equation}
\label{eqn:multipoperators}
\txtd U= (p\;\Id+T_f) U~\txtd t + \txtd W,\qquad U=U(t).  
\end{equation}
Set $M:=\sup \{x : x \in \textrm{essran}(f) \}=\textrm{esssup}(f)$. For the 
operator $p\;\Id+T_f$, the spectrum is given by the set $p + \textrm{essran}(f)$. 
It is contained in the left half of the complex plane as long as $p + M < 0$. 
Therefore, the associated dynamical system undergoes a bifurcation at $p^*=-M$. 
We compute the norm of the variance as in the previous section assuming 
$p < p^*=-M$
\begin{align*}
\Vert V_\infty \Vert_{B(H)}^2 &= \Big \Vert \int_0^\infty 
\txte^{2r (p+T_f)} ~\txtd r \Big \Vert_{B(H)}^2 \\
&= \sup_{h \in D(T_f), \Vert h \Vert =1} \Big \Vert \int_0^\infty 
\txte^{2r (p+f)} h ~\txtd r \Big \Vert_{H}^2 \\
&\overset{p < - M }{=} \sup_{h \in D(T_f), \Vert h \Vert =1} 
\Big \Vert \frac{h}{2(p+f)} \Big \Vert_{H}^2  \\
&= \sup_{h \in D(T_f), \Vert h \Vert =1} \int_\cX \Big | 
\frac{h(x)}{2(p+f(x))} \Big |^2~\txtd \mu (x)
\end{align*}

Now we look at the spectrum of $f$ crossing the 
imaginary axis (i.e.~$p \rightarrow -M^-$). The divergence of the 
last expression depends of course on the function $f$ and on the 
$L^2$ space we consider. Suppose we set $\mu$ equal to the Lebesgue 
measure and $\cX \subset \mathbb{R}$, which is the most straightforward 
generalization of the Laplacian case. To simplify the treatment we also 
assume that $f$ attains its essential supremum at only one point 
$x^* \in \cX$ and that it is continuous in a neighborhood of that 
point. We assume continuity around $x^*$ in order for the limit 
$\lim_{x \rightarrow 0^-} f(x)$ to be well-defined and independent 
of the sequence converging to $0^-$. Moreover, without loss of 
generality we set $x^*=0$ and $p^*=0$. By definition of $M$ and 
continuity of $f$, we have $\lim_{x \rightarrow 0} f(x) = 0$. Assume 
now $\mu(\cX) < \infty$. Let $\theta(x)$ be a smooth function 
such that $\lim_{x \rightarrow 0} \frac{f(x)}{\theta(x)}=1$ and 
that $\theta$ is bounded from above and below outside any 
neighborhood of $x^*$ (intuitively $\theta$ represents the order 
of $f$ at $x^*=0$). Then the function $h$ defined by 
$h(x)=x^{-1/2}\theta(x)$ is in $L^2=L^2(\cX,\mu)$ and $\theta$ can be chosen 
so that $h$ has unit norm. We obtain
\begin{align*}
\Vert V_\infty \Vert_{B(L^2)}^2
&= \sup_{h \in D(T_f), \Vert h \Vert =1} \int_\cX \Big | 
\frac{h(x)}{2(p+f(x))} \Big |^2~\txtd x \\
& \geq \int_X \Big | \frac{\theta(x)}{2(p+f(x))x^{1/2}} \Big |^2~\txtd x
\end{align*}
and taking the limit inferior as before
\begin{align*}
\liminf_{p \rightarrow 0^-} \Vert V_\infty \Vert_{B(H)}^2
& \geq \liminf_{p \rightarrow 0^-} \int_\cX \Big | 
\frac{\theta(x)}{2(p+f(x))x^{1/2}} \Big |^2~\txtd x \\
& \overset{\textnormal{Fatou}}{\geq} \int_X \frac{\theta^2(x)}{f^2(x)}
 \frac{1}{4x}~\txtd x = + \infty.
\end{align*}

\begin{rem}
The assumption $\mu(X) < + \infty$ can be relaxed by multiplying $h$ by 
a function $g$ such that $g/f$ is square integrable in a neighborhood of 
infinity. Requiring that the essential supremum is attained at a unique 
point $x^*$ can also be easily avoided studying each of the points 
separately. In any case, for each of them the analysis is similar.  
\end{rem}

We have shown the following:
\begin{theorem}
Consider a map $f: \mathbb{R} \rightarrow \mathbb{C}$ and the 
stochastic evolution equation~\eqref{eqn:multipoperators}
over the domain $D(T_f) \subset L^2(\mathbb{R},\mu)$ for some 
sigma-finite measure $\mu$. Assume $f$ to be 
continuous in a neighborhood of the points at which it attains its 
essential supremum. Then
\begin{equation}
\lim_{p \rightarrow -\textrm{esssup}(f)^-} \Vert V_\infty 
\Vert_{B(L^2)} = + \infty.
\end{equation}
\end{theorem}

\subsection{The general case}

We have seen, how the assumption of discrete spectrum allows for identifying 
the direction along which the covariance operator diverges. We have argued 
that such an approach cannot be used in the general setting because it involves 
considering eigenfunctions for the operator $A$. Indeed, if the operator
has continuous spectrum, eigenfunctions do not exist. Nevertheless, 
one can find ``approximate'' eigenfunctions for 
elements at the boundary of the spectrum. 

\begin{theorem}[Weyl's criterion]
\label{thm:weyl}
Consider a closed linear operator
$\cA$ on a Hilbert space $H$.
If $\lambda \in \partial \textnormal{spec}(\cA)$ there exists a sequence 
$\{ u_k \}_{k\in \mathbb{N}}$ in $H$ such that $\Vert u_k \Vert = 1$ 
and $$ \lim _{k\to \infty }\Vert \cA u _{k}-\lambda u_{k} \Vert =0.$$
\end{theorem}

For completeness, and since it might not be well-known, we provide a 
proof of Weyl's criterion. We denote by $\textnormal{res}(\cA):=\mathbb{C} 
\, \backslash \, \textnormal{spec}(\cA)$ the resolvent set of $\cA$. We
need the following auxillary result: 

\begin{lemma}
\label{lem:weyl}
If $\cA$ is a closed operator and $z \in \textnormal{res}(\cA)$, 
then $\Vert (A- z)^{-1} \Vert \geq \textnormal{dist}(z, \textnormal{spec}(\cA))^{-1}.$
\end{lemma}

\begin{proof} Fix $z \in \textnormal{res}(\cA)$ 
and $\lambda \in \textnormal{spec}(\cA)$. Then $|z - \lambda| \geq 
\Vert (\cA - z)^{-1} \Vert^{-1}$ (indeed, by Proposition 2.9 in~\cite{Schmuedgen}, 
if $\lambda$ satisfies $|z - \lambda| < \Vert (\cA - z)^{-1} \Vert^{-1}$ 
then $\lambda \in \textnormal{res}(\cA)$). This also implies 
$$\textnormal{dist}(z, \textnormal{spec} (\cA)) = \inf_{\lambda \in 
\textnormal{spec} (\cA)} |z - \lambda| \geq \frac{1}{\Vert (\cA - z)^{-1} \Vert}$$ 
so that rearranging gives the claimed result.
\end{proof}

\begin{proof}[Proof of Theorem~\ref{thm:weyl}]
Since $\lambda \in \partial \textnormal{spec}(\cA)$ we can find a 
sequence $\{\lambda_k \} \subset \textnormal{res} (\cA)$ such that 
$\lambda_k \rightarrow \lambda$. For each $\lambda_k$ we can apply Lemma~\ref{lem:weyl} 
and for each $k$ $$\Vert (A- \lambda_k)^{-1} \Vert \geq 
\textnormal{dist}(\lambda_k, \textnormal{spec}(\cA))^{-1}. $$ We can also 
find $v_k \in D ((A- \lambda_k)^{-1}) $ such that $\Vert v_k \Vert = 1$ 
and $$\Vert (A- \lambda_k)^{-1} v_k \Vert \geq \frac{1}{2}\textnormal{dist}(
\lambda_k, \textnormal{spec}(\cA))^{-1}. $$ Then, the sequence $$u_k 
:= \frac{(A- \lambda_k)^{-1} v_k}{\Vert (A- \lambda_k)^{-1} v_k \Vert}$$ 
is normalized and satisfies the claim. Indeed one computes
\begin{align*}
\Vert \cA u _{k}-\lambda u_{k} \Vert &= \Vert \cA u _{k}-\lambda_k u_{k} 
+ \lambda_k u _{k}-\lambda u_{k} \Vert  \leq \Vert \cA u _{k}-\lambda_k 
u_{k} \Vert + |\lambda_k -\lambda| \Vert u_{k} \Vert =& \\
&=\frac{\Vert v_k \Vert}{\Vert (A- \lambda_k)^{-1} v_k \Vert} + 
|\lambda_k -\lambda| \Vert u_{k} \Vert
= \frac{1}{\Vert (A- \lambda_k)^{-1} v_k \Vert} + |\lambda_k -\lambda| 
\leq \\
& \leq 2 \textnormal{dist}(\lambda_k, \textnormal{spec}(\cA)) + 
|\lambda_k -\lambda| \rightarrow 0.
\end{align*}
Therefore, Weyl's criterion follows.
\end{proof}

Since our bifurcation point necessarily involves spectrum on the boundary, 
we can exploit this theorem to prove a result similar to the one obtained 
in the discrete setting. 

\begin{theorem}
\label{thm:generalspectrum}
Consider the stochastic evolution equation
\begin{equation}
\label{eqn:generalspectrum}
\txtd U= (p\;\Id+\cA) U~\txtd t + \sigma B~\txtd W
\end{equation}
and assume $\textnormal{spec}(\cA)=\{\lambda_* \} \cup 
\textnormal{spec}_-(\cA)$, with $\Re(\lambda_*)=0$ and 
$\textnormal{spec}_-(\cA) \subset \{z \in \mathbb{C} : \Re(z)<0 \}$. 
Also assume that $BQB^* \geq c > 0$; in particular this holds for 
$B=Q=Id$. Then there exists a sequence $\{u_k \}_{k \in \mathbb{N}} 
\subset H$ such that for each $k$
\begin{equation}
\lim_{p \rightarrow 0^-} \langle V_\infty u_k, u_k \rangle = + \infty.
\end{equation}
\end{theorem}

\begin{proof}
Since $\lambda_* \in \partial \textnormal{spec}(\cA)$, by Weyl's criterion 
there exists a sequence $\{u_k \}_{k \in \mathbb{N}}$ s.t.~
\benn
\lim _{k\to \infty } \Vert \cA u_{k}-\lambda_* u_{k} \Vert =0, 
\; \; \Vert u_k \Vert = 1. 
\eenn
Our aim is to find a subsequence of $\{u_k\}$ that satisfies the claim. 
Define $e_k := \cA u_{k}-\lambda_* u_{k}$, $\bar{e}_k := \cA^* u_{k}-
\bar{\lambda}_* u_{k}$. Note that we have
\benn
\lim _{k\to \infty } \Vert e_{k} \Vert =0, \; \;
 \lim _{k\to \infty } \Vert \bar{e}_{k} \Vert =0. 
\eenn
As usual, the Lyapunov equation gives
\begin{equation*}
\langle (p+\cA)V_\infty u_k,u_k \rangle 
+ \langle V_\infty (p+\cA)^* u_k, u_k \rangle 
= - \sigma ^2 \langle BQB^* u_k, u_k \rangle.
\end{equation*}
Then, we can compute:
\beann
\langle (p+\cA)V_\infty u_k,u_k \rangle + \langle V_\infty (p+\cA)^* 
u_k, u_k \rangle &=&
2p \langle V_\infty u_k,u_k \rangle + \lambda_* \langle V_\infty 
u_k,u_k \rangle \\
&&+ \langle V_\infty u_k, e_k \rangle + \bar{\lambda}_* \langle 
V_\infty u_k, u_k \rangle + \langle \bar{e}_k, V_\infty u_k \rangle \\
&=& 2p \langle V_\infty u_k, u_k \rangle + \langle V_\infty u_k, 
e_k \rangle + \langle \bar{e}_k, V_\infty u_k \rangle.
\eeann
Therefore, we conclude that
\benn
\langle V_\infty u_k, u_k \rangle = \frac{ - \langle V_\infty u_k, 
e_k \rangle - \langle \bar{e}_k, V_\infty u_k \rangle - \sigma ^2 
\langle BQB^* u_k, u_k \rangle}{2p }.
\eenn
We will now consider two cases: first assume $\Vert V_\infty u_k \Vert < C$ 
for some $C > 0$ (uniformly in $k$). Then, by eventually discarding some of 
the pairs $(e_k, \bar{e}_k)$, we can also assume the bounds 
\begin{equation*}
\Vert e_k \Vert < \frac{1}{k}, \; \; \; \Vert \bar{e}_k \Vert < \frac{1}{k}
\end{equation*}
Together, this gives
$$ | \langle V_\infty u_k, e_k \rangle | \leq \Vert V_\infty 
u_k \Vert \Vert e_k \Vert \leq \frac{C}{k}.$$
And therefore:
\begin{align*}
\langle V_\infty u_k, u_k \rangle & \leq \frac{ 
| \langle V_\infty u_k, e_k \rangle | + | \langle 
\bar{e}_k, V_\infty u_k \rangle | - \sigma ^2 
\langle BQB^* u_k, u_k \rangle}{2p } \\
&\leq \frac{ 2C/k - \sigma ^2 \langle BQB^* u_k, 
u_k \rangle}{2p } \leq \frac{ 2C/k - \sigma ^2 c}{2p}.
\end{align*}
The sequence $\{u_k\}_{k > 2C/\sigma^2 c}$ satisfies the claim 
of the theorem. In this case the stronger result 
\benn
\langle V_\infty u_k, u_k \rangle = \cO\left(\frac{1}{p} \right), \quad
\textrm{ as } p \rightarrow 0^-
\eenn
holds. Suppose now $\Vert V_\infty u_k \Vert < C$ does not hold for 
any $C > 0$: then, there exists a subsequence $\{u_{k_j}\} \subset 
\{u_k\}$  such that $\Vert V_\infty u_{k_j} \Vert \geq j$. But this 
implies that also $ \langle V_\infty u_{k_j}, u_{k_j} \rangle = 
\Vert \sqrt{V_\infty} u_{k_j} \Vert^2 \rightarrow \infty 
\textrm{ as } j \rightarrow \infty $, which implies the claim. 
\end{proof}

We conclude with some remarks. Note that, in the previous sections, we 
were not only able to prove that some component of the variance diverges, 
but also to compute its rate of divergence. In the last proof we had to 
exclude the possibility that $\Vert V_\infty u_k \Vert$ has a divergent 
subsequence. If such a subsequence exists, the variance still diverges, 
but its rate of divergence is in general unknown. Therefore, we cannot 
conclude as before. Nevertheless, we observe that if we replace the 
genericity condition $\langle BQB^* u_{k^*}, u_{k^*} \rangle \neq 0 $ 
with the assumption 
\benn
\langle V_\infty u_k, e_k \rangle + \langle \bar{e}_k, V_\infty 
u_k \rangle + \sigma ^2 \langle BQB^* u_k, u_k \rangle \neq 0 \quad 
\textrm{ for infinite values of } k, 
\eenn
we can indeed conclude that there exists a sequence such that 
\benn
\langle V_\infty u_k, u_k \rangle = \cO\left(\frac{1}{p} \right),\quad
 \textrm{ as } p \rightarrow 0^-.
\eenn
Furthermore, Theorem~\ref{thm:generalspectrum} is actually stronger than the 
results obtained previously. Indeed it shows that, under some assumptions, there 
exists a \emph{whole sequence} of ``approximate eigenfunctions'' such that 
the components of the covariance operator along this sequence diverge. Of 
course, such a sequence might be constant, as it is the classical case for 
discrete spectrum, which occurs for many differential operators on bounded 
domains.

As another remark, suppose we only look for one vector $u \in H$ such 
that 
\benn
\langle V_\infty u, u \rangle = \cO\left(\frac{1}{p} \right), \quad
\textrm{ as } p \rightarrow 0^-.
\eenn
Then, to obtain the claim on the asymptotic limit of the covariance operator, 
we only have to require that there exists a $k$ such that 
\benn
\langle V_\infty u_k, e_k \rangle + \langle \bar{e}_k, 
V_\infty u_k \rangle + \sigma ^2 \langle BQB^* u_k, u_k \rangle \neq 0,
\eenn
which is a much weaker condition. 

\bibliographystyle{plain}
\bibliography{../my_refs}

\end{document}